\documentclass[12pt]{amsart}

\usepackage{latexsym}
\usepackage{amssymb}
\usepackage{amsmath}
\usepackage{amsfonts}

\usepackage{setspace,  
hyperref, epigraph} 

\usepackage{breakurl}   

\usepackage{enumitem} 

\newtheorem{thm}{Theorem}[section]
         
\newtheorem{prop}[thm]{Proposition}    
   
\newtheorem*{rlem}{Ramsey's Theorem} 
\newtheorem*{slem}{Szemer\'{e}di's Theorem}    
\newtheorem*{plem}{Peano's Theorem}     

\theoremstyle{definition}
\newtheorem{defn}[thm]{Definition}

\newtheorem*{conjecture}{Conjecture}

\newcommand{\bbU}{\mathbb{U}}
\newcommand{\bbV}{\mathbb{V}}

\newcommand{\bbH}{\mathbb{H}}

\newcommand{\bbM}{\mathbb{M}}
\newcommand{\bbI}{\mathbf{I}}

\newcommand{\bbP}{\mathbb{P}}
\newcommand{\bbQ}{\mathbb{Q}}
\newcommand{\bbR}{\mathbb{R}}

\newcommand{\bbN}{\mathbb{N}}
\newcommand{\bbS}{\mathbb{S}}

\newcommand{\cG}{\mathcal{G}}

\newcommand{\cM}{\mathcal{M}}

\newcommand{\cP}{\mathcal{P}}

\newcommand{\Ul}{\mathfrak{U}}

\newcommand{\st}{\operatorname{\mathbf{st}}}
\newcommand{\sr}{\operatorname{\mathbf{sr}}}
\newcommand{\sha}{\operatorname{\mathbf{sh}}}

\newcommand{\dom}{\operatorname{dom}}
\newcommand{\fin}{\operatorname{\mathbf{fin}}}

\newcommand{\rank}{\operatorname{rank}}
\newcommand{\ra}{\rangle}
\newcommand{\la}{\langle}
\newcommand{\eqi}{\; \leftrightarrow \;}
\newcommand{\al}{\alpha}

\newcommand{\ep}{\varepsilon}

\newcommand{\om}{\omega}

\newcommand{\ZFC}{\textbf{ZFC}}

\newcommand{\IST}{\textbf{IST}}
\newcommand{\HST}{\textbf{HST}}

\newcommand{\GRIST}{\textbf{GRIST}}
\newcommand{\SPOT}{\mathbf{SPOT}}

\newcommand\astx{{}^\ast\!X}

\newcommand{\infy}{\operatorname{\mathbf{aa}}}

\newcommand{\ZF}{\mathbf{ZF}}
\newcommand{\SCOT}{\mathbf{SCOT}}
\newcommand{\SPOTS}{\mathbf{SPOTS}}
\newcommand{\SCOTS}{\mathbf{SCOTS}}

\newcommand{\ACC}{\mathbf{ACC}}
\newcommand{\ADC}{\mathbf{ADC}}
\newcommand{\SP}{\mathbf{SP}}
\newcommand{\SC}{\mathbf{SC}}
\newcommand{\SN}{\mathbf{SN}}
\newcommand{\CC}{\mathbf{CC}}
\newcommand{\DC}{\mathbf{DC}}

\newcommand{\AC}{\mathbf{AC}}

\newcommand{\T}{\mathbf{T}}

\newcommand{\bs}{\operatorname{\mathbf{ir}}}

\newcommand{\wpi}{\Pi}
\newcommand{\uhr}{\upharpoonright}
\newcommand{\sfa}{\mathsf{a}}
\newcommand{\sfc}{\mathsf{b}}
\newcommand{\sfm}{\mathsf{m}}
\newcommand{\sfn}{\mathsf{n}}

\begin{document}\title {Multi-level Nonstandard Analysis and the Axiom of Choice
}

\author{Karel Hrbacek}
\address{Department of Mathematics\\                City College of CUNY\\                New York, NY 10031\\}                \email{khrbacek@icloud.com}

\date{August 25, 2024}

\begin{abstract}
Model-theoretic frameworks for Nonstandard Analysis depend on the  existence of nonprincipal ultrafilters, a strong form of the Axiom of Choice ($\AC$). Hrbacek and Katz, APAL 72 (2021) formulate axiomatic nonstandard set theories $\SPOT$ and $\SCOT$ that are conservative extensions of respectively $\ZF$ and $\ZF + \ADC$ (the Axiom of Dependent Choice), and in which a significant part of Nonstandard Analysis can be developed. 
The present paper extends these theories to  theories with many levels of standardness, called respectively
$\mathbf{SPOTS}$ and $\mathbf{SCOTS}$.  It shows that Jin's recent nonstandard proof of Szemer\'{e}di's Theorem can be carried out in $\mathbf{SPOTS}$. The theories $\mathbf{SPOTS}$ and $\mathbf{SCOTS}$ are conservative over $\ZF + \ADC$.

\end{abstract}

\maketitle

\emph{This version corrects the  formulation of Proposition 3.3 (Factoring Lemma) on page 10 and revises  the notion of admissible formula (Definition 4.1 on page 11). While the new formulations suffice for the proofs in Sections 5 and 6,  Proposition 4.5  on page 13 of the previous version cannot be proved in $\SCOTS$ and the previous  Section 7,  which depends on this proposition in an essential way, is unjustified.}

\bigskip
\emph{The claim that $\SPOTS$ is conservative over $\ZF + \ACC$ made in Theorem 4.7 
is also unproved. The argument given for it in Subsection~8.5 fails because in the absence of $\ADC$ the forcing with $\bbP$ over $\cM$ might add new countable sets. This invalidates the Factoring Lemma and hence the claim that $\mathbf{HO}$ holds in $\cM_\infty$ made in the paragraph before Proposition~8.6.
}

\bigskip
\emph{This material is omitted in the current version.}

\newpage
\tableofcontents

\section{Introduction}

Nonstandard Analysis  is sometimes criticized for its implicit dependence on the Axiom of Choice ($\AC$) (see eg  Connes~\cite{C}\footnote{
Detailed examination of  Connes's views is carried out in Kanovei, Katz and Mormann~\cite{KKM}, Katz and Leichtnam~\cite{KL} and Sanders~\cite{Sa2}.
}).
Indeed, model-theoretic frameworks 
based on hyperreals require the existence of nonprincipal ultrafilters over $\mathbb{N}$, a strong form of $\AC$:
If~$\ast$ is the mapping that assigns to each $X \subseteq \bbN$ its nonstandard extension $\astx$, and if $\nu \in {}^\ast \bbN \setminus \bbN$ is an unlimited integer, then the set
$U = \{ X \subseteq \bbN \,\mid\, \nu \in \astx\}$ is a nonprincipal ultrafilter over $\bbN$.

The common axiomatic/syntactic frameworks for nonstandard methods (see Kanovei and Reeken~\cite{KR}), such as $\IST$ or $\HST$, 
include $\ZFC$ among their axioms.
The dependence on $\AC$ cannot be avoided by simply removing it from the list of axioms (see Hrbacek~\cite{H1}).
These theories postulate some version of the \emph{Standardization Principle}:

 \emph{For every formula $\Phi(x)$ in the language of the theory (possibly with additional parameters) and every standard set $A$ 
there exists a standard set $S$ such that for all standard~$x$}
$x \in S \eqi x\in A \,\wedge\, \Phi(x).$

This set is denoted \; ${}^{\st} \{x \in A \mid \Phi(x)\}$.
It follows that, for an unlimited $\nu \in \bbN$,  the standard set \;$U = {}^{\st} \{ X \in  \cP(\bbN) \,\mid\, \nu \in X\}$ is a nonprincipal ultrafilter over $\bbN$.
 
\bigskip
\emph{All work in Nonstandard Analysis based on these two familiar frameworks thus depends essentially on the Axiom of Choice.}\footnote{
Nonstandard Analysis that does not use $\AC$, or uses only weak versions of it,  can be found in  the work of Chuaqui, Sommer and Suppes (see eg~\cite{SS}), in papers on Reverse Mathematics of Nonstandard Analysis (eg Keisler~\cite{K3},  Sanders~\cite{Sa}, van den Berg et al~\cite{VBS} and others), and in the work  of Hrbacek and Katz~\cite{HK, HK2, HK3} and the present paper, based on $\SPOT$/$\SCOT$.}

\bigskip
While strong forms of $\AC$ (Zorn's Lemma, Prime Ideal Theorem) are instrumental in many abstract areas of mathematics, such as general topology (the product of compact spaces is compact), measure theory (there exist sets that are not Lebesgue measurable) or functional analysis (Hahn-Banach theorem), it is undesirable to have to rely on them for results in ``ordinary'' mathematics such as calculus, finite  combinatorics and  number theory.\footnote{ The issue is not the validity of such results but the method of proof. It is a consequence of Shoenfield's Absoluteness Theorem (Jech~\cite{J}, Theoerm 98) that all $\Pi^1_4$ sentences of second-order arithmetic that are provable in $\ZFC$ are also provable in $\ZF$. Most theorems of 
number theory and real analysis (eg Peano's Theorem; see  Hanson's answer in~\cite{MO}) can be formalized as $\Pi^1_4$ statements. 
But the $\ZF$ proofs obtained from $\ZFC$ proofs by this method are far from ``ordinary.''}
 
Hrbacek and Katz~\cite{HK}
introduced nonstandard  set theories $\SPOT$ and $\SCOT$.
In order to avoid the reliance on $\AC$, Standardization needs to be weakened.
The theory  $\SPOT$ has three simple axioms: Standard Part, Nontriviality and Transfer. 
It is a subtheory of  the  better known nonstandard set theories $\IST$ and $\HST$, but unlike them, it is a conservative extension of $\ZF$. Arguments carried out in $\SPOT$ thus do not depend on any form of $\AC$. 
Infinitesimal analysis can be conducted in $\SPOT$ in the usual way. 
It only needs to be verified that any use of Standardization 
can be justified by the special cases of this principle that are available in $\SPOT$.

Traditional proofs in ``ordinary'' mathematics either do not use $\AC$ at all, or refer only to its weak forms, notably the Axiom of Countable Choice ($\ACC$) or the stronger Axiom of Dependent Choice ($\ADC$). These axioms are generally accepted and often used without comment. 
They are necessary to prove eg 
the equivalence of the $\ep$--$\delta$ definition and the sequential definition of continuity at a given point for functions $f: X \subseteq \bbR \to \bbR$, or the countable additivity of Lebesgue measure,
but they do not imply the strong consequences of $\AC$ such as the existence of nonprincipal ultrafilters or the Banach--Tarski paradox.
The theory $\SCOT$ is obtained by adding to $\SPOT$ the external version of the Axiom of Dependent Choice; it is a conservative extension of $\ZF+ \ADC$.

Nonstandard Analysis with \emph{multiple levels of standardness} has been used in combinatorics and number theory by  Renling Jin, Terence Tao, Mauro Di Nasso and others.
Jin~\cite{RJ}  recently gave a groundbreaking nonstandard proof of Szemer\'{e}di's Theorem
 in a model-theoretic framework that has three levels of infinity.
\begin{slem}
If $D\subseteq \bbN$ has a positive upper density, then $D$ contains a $k$-term arithmetic progression for every $k \in \bbN$.
\end{slem}

The main objective of this paper is to extend the above two  theories to theories $\SPOTS$ and $\SCOTS$ with many levels of standardness and consider their relationship  to $\AC$.
An outline of $\SPOT$ and $\SCOT$ is given in Section~\ref{spot}.  
Section~\ref{ultrapower} reviews the familiar properties of ultrapowers and iterated ultrapowers in a form suitable 
to motivate  multi-level versions of these theories, which are formulated in Section~\ref{spots}.
The next two sections illustrate various ways to use multiple levels of standardness.
In Section~\ref{ramsey} Jin's proof of Ramsey's Theorem is formalized in $\mathbf{SPOTS}$, and Section~\ref{szemeredi} explains how Jin's proof of Szemer\'{e}di's Theorem can be developed in it.  
Finally, in Section~\ref{conservativity} it is proved  that $\SCOTS$ is a conservative extension of $\ZF + \ADC$ and that $\SPOTS$ is conservative over $\ZF + \ACC$.
It is an open problem whether $\SPOTS$ is a conservative extension of $\ZF$.


\section{Theories \textbf{SPOT} and \textbf{SCOT}} \label{spot}

By an $\in$--language we mean the language that has a primitive binary  membership predicate $\in$. 
The classical theories $\ZF$ and $\ZFC$ are formulated in the $\in$--language.  
It is enriched by defined symbols for constants, relations,  functions and operations customary in traditional mathematics. 
For example, it contains names $\bbN$ and $\bbR$ for the sets of natural and real numbers; they are viewed as defined in the traditional way ($\bbN$ is the least inductive set, $\bbR$ is defined in terms of Dedekind cuts or Cauchy sequences). 

Nonstandard set theories, including  $\SPOT$ and $\SCOT$, are formulated in the $\st$--$\in$--language. They  add to the $\in$--language a unary predicate symbol $\st$, where $\st(x)$ reads ``$x$ is standard,'' and possibly other symbols.
They postulate that standard infinite sets contain also nonstandard elements. For example, $\bbR$ contains infinitesimals and unlimited reals, and $\bbN$ contains unlimited natural numbers.

We use $\forall$ and $\exists$ as quantifiers over sets and $\forall^{\st}$
and $\exists^{\st}$ as quantifiers over standard sets.  
All free variables of a formula $\Phi(v_1,\ldots,v_k)$ are assumed to be among  $v_1,\ldots,v_k$ \emph{unless} explicitly specified otherwise. This is usually done informally by saying that the formula has \emph{parameters} (ie, additional free variables), possibly restricted to objects of a certain kind).

The axioms of $\SPOT$  are:

$\ZF$ (Zermelo - Fraenkel Set Theory; Separation and Replacement schemata apply to $\in$--formulas only).

\bigskip
$\T$ (Transfer) 
Let $\phi(v)$ be an $\in$--formula with standard parameters. Then
$$\forall^{\st} x\; \phi(x)  \to  \forall x\; \phi(x).$$

\bigskip
$\mathbf{O}$ (Nontriviality)  \quad
$\exists \nu \in \bbN\; \forall^{\st} n \in \bbN\; (n \ne \nu)$.

\bigskip
$\SP$  (Standard Part)
$$ \forall  A \subseteq \bbN \; \exists^{\st} B \subseteq \bbN \; \forall^{\st}  n \in \bbN \;
(n \in B \eqi n \in A). $$

We state some general results provable in $\SPOT$ (Hrbacek and Katz~\cite{HK}).

\begin{prop} \label{HKP1}
\emph{Standard natural numbers precede all nonstandard natural numbers:} \quad
$$\forall^{\st} n \in \bbN \; \forall m \in \bbN \;(m < n \to \st(m) ).$$
\end{prop}

\begin{prop} \emph{(Countable Idealization)}
Let $\phi(u,v)$ be an $\in$--formula with arbitrary parameters.
$$\forall^{\st}  n \in \bbN\;\exists x\; \forall m \in \bbN\; (m \le n \;\to  \phi(m,x))\eqi 
\exists x \; \forall^{\st} n \in \bbN \; \phi(n,x).$$
\end{prop}

The  dual form of Countable Idealization is
$$\exists^{\st}  n \in \bbN\;\forall x\; \exists m \in \bbN\; (m \le n \,\wedge\,  \phi(m,x))
\eqi \forall x \; \exists^{\st} n \in \bbN \; \phi(n,x) .$$
 
Countable Idealization easily implies the following more familiar form. We use  $\forall^{\st \fin}$ and $\exists^{\st \fin}$ as quantifiers
over standard finite sets.

 \label{corcountideal}
\emph{Let $\phi(u,v)$ be an $\in$--formula with arbitrary parameters.  
 For every standard countable
set}~$A$
\[
\forall^{\st \fin} a \subseteq A \, \exists x \, \forall y \in a\;
\phi(x,y) \eqi \exists x\, \forall^{\st} y \in A\; \phi(x,y) .
\]

The axiom $\SP$ is often used in the equivalent
form
\begin{equation}
 \forall x \in \bbR \;(x \text{ limited } \to \exists^{\st} r \in \bbR \;(x \simeq r)) \tag{$\SP'$}.
\end{equation}
We recall that $x$ is \emph{limited} iff $|x| \le n$ for some standard $n \in \bbN$, and $x \simeq r$ iff $|x - r| \le 1/n $ for all standard $n \in \bbN$, $n \neq 0$; $x$ is \emph{infinitesimal} if $x \simeq 0 \,\wedge\, x \neq 0$. 
The unique standard real number $r$ is called the \emph{standard part of} $x$ or the \emph{shadow of} $x$; notation $r = \sha(x)$.
 
The axiom $\SP$
is also equivalent to Standardization over countable sets for $\in$--formulas (with arbitrary parameters):

\emph{Let $\phi(v)$ be an $\in$--formula with arbitrary parameters. Then}
\begin{equation}
\exists^{\st} S\; \forall^{\st} n \; (n \in S \eqi n \in \bbN
\,\wedge\, \phi(n)).  \tag{$\SP''$}
\end{equation}

\begin{proof}
Let $A = \{ n \in \bbN \,\mid\, \phi(n) \}$ and apply  $\SP$.
\end{proof}

The ``nonstandard'' axioms of $\SPOT$  extend to $\ZF$ the insights of Leibniz about real numbers (see Bair et al~\cite{BBE, 20z}, Katz and Sherry~\cite{KS} and Katz, Kuhlemann and Sherry~\cite{KKS}):

\begin{itemize}
\item
Assignable vs inassignable distinction 
[standard vs nonstandard]
\item
Law of continuity [Transfer]
\item
Existence of infinitesimals [Nontriviality]
\item
Equality up to infinitesimal terms that need to be discarded  
[Standard Part]
\end{itemize}
This can be taken as a justification of the axioms of $\SPOT$ independent of the proof of its conservativity ove $\ZF$.

\bigskip
The scope of the axiom schema $\mathbf{SP''}$ can be extended.

\begin{defn}
An $\st$--$\in$--formula  $\Phi(v_1, \ldots,v_r)$ is \emph{special} if it is of the form 
$$\mathsf{Q}^{\st} u_1\ldots \mathsf{Q}^{\st} u_s \,\psi(u_1,\ldots ,u_s,v_1, \ldots,v_r)$$
 where $\psi $ is an $\in$--formula and each $\mathsf{Q}$ stands for $\exists$ or $\forall $.

We use $\forall^{\st}_{\bbN} \,u\ldots$ and $\exists^{\st}_{\bbN}\, u \ldots$ as shorthand for, respectively,
$\forall^{\st} u\,( u \in \bbN \to \ldots)$ and $\exists^{\st} u \,(u \in \bbN  \,\wedge\, \ldots)$. 

An $\bbN$-\emph{special} formula is a formula  of the form 
\[
\mathsf{Q}^{\st}_{\bbN} u_1\ldots \mathsf{Q}^{\st}_{\bbN} u_s \,\psi(u_1,\ldots u_s, v_1, \ldots,v_r)
\]
where $\psi$ is an $\in$--formula.
\end{defn}
 
\begin{prop}$(\SPOT)$ \emph{ (Countable Standardization for $\bbN$-Special Formulas)}\label{cssf}
Let $\Phi(v)$ be an $\bbN$-special formula with arbitrary parameters. Then
\[
\exists^{\st} S\; \forall^{\st} n \; (n \in S \eqi n \in \bbN
\,\wedge\, \Phi(n)).
\]
\end{prop}
Of course, $\bbN$ can be replaced by any standard countable set.

\begin{proof}
We give the argument for a typical case
$$\forall^{\st}_\bbN  u_1\,\exists^{\st}_\bbN u_2\,\forall^{\st}_\bbN u_3 \;\psi(u_1,u_2,u_3, v).$$
By $\SP''$ there is a standard set $R$ such that for all standard $n_1, n_2, n_3, n$
$$\la  n_1, n_2, n_3, n \ra \in R \eqi \la n_1, n_2, n_3, n \ra \in \bbN^4 \,\wedge\, \psi(n_1, n_2, n_3,n).$$ 
We let $R_{n_1, n_2, n_3} = \{n \in \bbN \,\mid\, \la n_1, n_2, n_3, n \ra \in R\}$ and 
$$S = \bigcap_{n_1 \in \bbN}\, \bigcup_{n_2 \in \bbN}\,\bigcap_{n_3 \in \bbN} \,R_{n_1, n_2, n_3} .$$
Then $S$ is standard and for all standard $n$:
\[
\begin{aligned}
n \in S & \eqi 
 \forall n_1 \in \bbN\,\exists n_2 \in \bbN\,\forall n_3 \in \bbN\; (n \in R_{n_1, n_2, n_3} )
\\&\eqi
\text{(by Transfer) }
\forall^{\st}_{\bbN} n_1 \,\exists^{\st}_{\bbN} n_2 \;\forall^{\st}_{\bbN} n_3 \; (n \in R_{n_1, n_2, n_3}) 
\\&\eqi
\text{ (by definition of $R$) }
  \forall^{\st}_{\bbN} n_1 \,\exists^{\st}_{\bbN} n_2 \;\forall^{\st}_{\bbN} n_3 \;\psi(n_1, n_2, n_3,n)
\\&\eqi 
\Phi(n). 
\end{aligned} 
\]
\end{proof}

Infinitesimal calculus can be developed in $\SPOT$  as far as the global version of Peano's Theorem;  
see Hrbacek and Katz~\cite{HK2, HK3}.

\begin{plem}
Let $F: [0, \infty)  \times \mathbb{R}  \to \mathbb{R}$ be a continuous function.
There is an interval $[0, a) $ with $0 < a \le \infty$
and a function $y : [0,a) \to \mathbb{R}$ such that
\[
\qquad y(0)= 0,\quad
\quad y'(x) =F(x,y(x)) \quad
\]
holds for all $ x \in [0, a)$, and if $a \in \mathbb{R}$ then $\lim_{x\to a^-} y(x) = \pm \infty$.
\end{plem}

We note that traditional proofs of the global version of Peano's Theorem use Zorn's Lemma or  the Axiom of Dependent Choice.

It is useful to extend $\SPOT$ by  two additional special cases of Standardization.

\bigskip
$\SN$  (Standardization for $\st$--$\in$--formulas with no parameters or, equivalently, with only standard parameters)
\emph{Let $\Phi (v)$ be an $\st$--$\in$--formula with standard parameters. Then}
$$\forall^{\st} A\, \exists^{\st}  S\; \forall^{\st} x\; (x \in S \eqi x \in A \,\wedge\,\Phi(x)).$$

\bigskip
\textbf{SF} (Standardization over standard finite sets) 
\emph{Let $\Phi(v)$ be an $\st$--$\in$--formula
with arbitrary parameters. 
Then}
$$ \forall^{\st \fin} A \;\exists^{\st} S\; \forall^{\st} x\;  ( x \in S \eqi
x \in A \; \wedge \; \Phi (x) ).$$

An important consequence of \textbf{SF} is the ability to carry out external induction.

\begin{prop}\label{extind} \emph{ (External Induction)}
Let $\Phi(v)$ be an $\st$--$\in$--formula with arbitrary parameters. Then $\mathbf{SPOT} + \mathbf{SF} $ proves the following:
$$[\,\Phi(0) \,\wedge\, \forall^{\st} n \in \bbN\, (\Phi(n) \to \Phi(n+1)) \,]\to \forall^{\st} m \,\Phi(m).  $$
\end{prop}

\begin{proof}
Let $m \in \bbN$ be standard. If $m = 0$, then $\Phi(m) $ holds.  Otherwise 
$\textbf{SF}$ yields a standard set $S \subseteq m$ such that $\forall^{\st} n < m\, (n \in S \eqi \Phi(n));$
clearly $0 \in S$. As $S$ is finite,  it has a greatest element $k$, which is standard by Transfer.
If $k < m$, then $k+1 \in S$, a contradiction. Hence $k = m$ and $\Phi(m)$ holds.
\end{proof}

\bigskip
$\SPOT^+$ is $\SPOT$ + \textbf{SN} + \textbf{SF}.

\emph{The theory }$\SPOT^+$ \emph{is a conservative extension of }$\ZF$.

\bigskip
This is proved for $\SPOT + \textbf{SN} $ in~\cite[Theorem B]{HK} (Propositions 4.15 and 6.7 there).
The  argument that \textbf{SF} can be also added conservatively over $\ZF$ is given at the end of Section~\ref{conservativity} (Proposition~\ref{sf}).

\bigskip
The theory $\SCOT$ is $\SPOT^+   + \DC$, where

\bigskip
  $\mathbf{DC}$ (Dependent Choice for $\st$--$\in$--formulas)
\emph{ Let $\Phi (u,v)$ be an $\st$--$\in$--formula with arbitrary parameters. 
If $\forall x \, \exists y \; \Phi(x,y)$, then for any $b$ there is a sequence $\la b_n \,\mid\, n \in \bbN\ra$ such that $b_0 = b $ and }
$\forall^{\st} n \in \bbN\; \Phi(b_n, b_{n+1})$.
 
\bigskip
Some general consequences of $\SCOT$ are (see~\cite{HK}):

$\CC\;$ (Countable $\st$--$\in$--Choice)
\emph{Let $\Phi (u,v)$ be an $\st$--$\in$--formula with arbitrary
parameters. Then} $\forall^{\st} n \in \bbN\; \exists x\; \Phi(n,x)
\to \exists f\, (f \text{ is a function} \,\wedge\, \forall^{\st} n
\in \bbN\, \Phi(n, f(n)).$

\bigskip
$\SC$ (Countable Standardization) 
\emph{Let $\Psi(v)$ be an
$\st$--$\in$--formula with arbitrary parameters. Then}
$$
\exists^{\st} S\; \forall^{\st} x \; (x \in S \eqi x \in \bbN
\,\wedge\, \Psi(x)).
$$
 
$\SCOT$ is a conservative extension of $\ZF + \ADC$ \cite[Theorem~5.10]{HK}.\footnote{
It is an open question whether $\SPOT^+ + \SC$ is a conservative extension of $\ZF$.} 
It allows such features as an
infinitesimal construction of the Lebesgue measure.  
It implies the axioms of Nelson's \emph{Radically Elementary Probability Theory}~\cite{N}.


\section{Ultrafilters, ultrapowers and iterated ultrapowers}\label{ultrapower}

In this section we review 
the construction of iterated ultrapowers in a form suitable for motivation and establishing the conservativity of the theories formulated in Section~\ref{spots}. 
We assume $\ZFC$, use classes freely and give no proofs. Some references for this material are Chang and Keisler~\cite{CK},  Enayat et al~\cite{E2} and Hrbacek~\cite{H4, H2b}.

Model theory deals with structures that are sets. For our purposes we need to construct ultrapowers of the entire set-theoretic universe $\bbV$. That means we have to deal with structures that are proper classes (eg $(\bbV, \in)$). We sometimes use the model-theoretic language and say that such structure \emph{satisfies a formula}, or that a mapping is an \emph{elementary embedding} of one class structure into another. It is well-known that the satisfaction relation $\vDash$ for such structures  cannot be proved to exist in $\ZFC$. But the concept makes sense for any one particular formula. Thus, if $\Ul =(\bbU; \bbV_1, \bbV_2,\ldots)$ is a (class) structure and $\Phi(v_1,\ldots,v_r)$ is a formula in the language of $\Ul$, we write
 $\Phi^\bbU(v_1,\ldots,v_r)$ for the formula obtained from $\Phi$ by restricting all quantifiers to $\bbU$,  ie, by replacing each occurrence of $\forall v$ by $\forall v \in \bbU$ and each occurrence of $\exists v$ by $\exists v \in \bbU$. [We usually abuse notation by not distinguishing between classes and their names in the language of $\Ul$.]
The statement that $\mathbf{J}$ is an elementary embedding of $\Ul^1$ to $\Ul^2$, for example,  means that, given any formula $\Phi$ of the appropriate language, $$\forall x_1,\ldots,x_r\in \bbU^1 \,(\Phi^{\bbU^1} (x_1,\ldots,x_r)\eqi  \Phi^{\bbU^2} (\mathbf{J}(x_1),\ldots,\mathbf{J}(x_r) )).$$

Let $U$ be an ultrafilter over $I$.
For $f, g \in \bbV^I$ we define
$$f =_U g \text{ iff } \{ i \in I \,\mid\, f(i) = g(i) \} \in U,$$
$$f \in_U g \text{  iff }\{ i \in I \,\mid\, f(i) \in g(i) \} \in U.$$

The usual procedure at this point is to form equivalence classes $[f]_U$ of functions  $f \in \bbV^I$ modulo $=_U$, using ``Scott's trick'' of taking only the functions of the minimal von Neumann rank to guarantee that the equivalence classes  are sets: 
$$ [f]_U = \{g \in \bbV^I\,\mid\, g =_U f \text{ and } \forall h\in \bbV^I\,
(h =_U f \to \rank h \ge \rank g)\};$$
see Jech~\cite[(9.3) and (28.15)]{J}.
Then
$\bbV^I/U =\{ [f]_U \,\mid\, f \in \bbV^I\}$, and 
$[f]_U \in_U [g]_U $ iff $f \in_U g$.

The \emph{ultrapower} of $\bbV$ by $U$ is the structure $(\bbV^I/U,  \in_U)$.

Let $\pi: I \to J$. 
Define the ultrafilter $V = \pi[U]$ over $J$ by
$$\pi[U] = \{ Y \subseteq J \,\mid\, \pi^{-1}[Y] \in U\}.$$

The mapping $\pi$ induces $\wpi: \bbV^J/V \to \bbV^I/U$ by $\wpi([g]_V) = [g\circ \pi]_U$.

The following proposition is an easy consequence of \L o\'{s}'s Theorem.

\begin{prop} \label{P1}
The mapping $\wpi $ is well-defined and it is an elementary embedding of $(\bbV^J\!/V , \in_V)$ into $(\bbV^I\!/U, \in_U) $.\\
In detail:  For any $\in$--formula $\phi$ and all $[f_1]_V,\ldots [f_r]_V \in \bbV^J/V$,
$$ \phi^{\bbV^J/V}([f_1]_V, \ldots, [f_r]_V ) \eqi  \phi^{\bbV^I/U}(\wpi([f_1]_V),\ldots, \wpi([f_r])_V ).$$
\end{prop}

The \emph{tensor product} of ultrafilters $U$ and $V$, respectively over $I$ and $J$, is the ultrafilter over $I \times J$ defined by
(note the order; Chang and Keisler~\cite{CK} use the opposite order):
$$Z \in U \otimes V \text{ iff } \{x \in I \,\mid\, \{ y \in J \,\mid\, \la x,y \ra \in Z\} \in V \} \in U.$$

The $n$-\emph{th tensor power} of $U$ is the ultrafilter over $I^n$ defined by recursion:
$$\otimes^{0} U = \{\{\emptyset\}\}; \quad \otimes^{1} U = U; \quad \otimes^{n+1} U = U \otimes (\otimes^n U).$$

\bigskip
In the following, 
$\sfa$, $\sfc$ range over finite subsets of $\bbN$.

If $|\sfa| = n$, let $\pi$ be the mapping of $I^n$ onto $I^\sfa$ induced by the order-preserving mapping of $ n$ onto $\sfa$. It follows that $U_\sfa = \pi[\otimes^n U]$ is an ultrafilter over $I^\sfa$.

For $\sfa \subseteq \sfc$ let $\pi^\sfc_\sfa$ be the restriction map of $I^\sfc$ onto $I^\sfa$ defined by
$\pi^\sfc_\sfa (\mathbf{i}) =  \mathbf{i} \upharpoonright \sfa$ for $\mathbf{i} \in I^\sfc$.

It is easy to see that $U_\sfa = \pi^\sfc_\sfa[U_\sfc]$.
We let $\bbV_\sfa =\bbV^{I^\sfa}/U_\sfa$ and write $[f]_\sfa$ for $[f]_{U_{\sfa}}$ and $\in_\sfa$ for 
$\in_{U_\sfa}$.
The mapping $\wpi^\sfc_\sfa $ induced by $\pi^\sfc_\sfa$ is an elementary embedding of $(\bbV_\sfa, \in_\sfa) $ into $(\bbV_\sfc , \in_\sfc)$.

\begin{prop} \label{P2}
If   $f \in \bbV^{I^\sfa}$,  $g \in \bbV^{I^\sfc}$ and
$\wpi^{\sfa \cup \sfc}_\sfa([f]_{\sfa})  =   \wpi^{\sfa \cup \sfc}_\sfc([g]_{\sfc})  $, then 
there is $h \in \bbV^{I^{\sfa \cap \sfc}}$ such that  
 $\wpi^{\sfa \cup \sfc}_{\sfa \cap \sfc} ([h]_{\sfa \cap \sfc})  =
\wpi^{\sfa \cup \sfc}_{\sfa} ([f]_{\sfa})  =   \wpi^{\sfa \cup \sfc}_{\sfc} ([g]_{\sfc})  $.
\end{prop}

Let $f, g \in \bigcup_{\sfa} \bbV^{I^\sfa} $; say  $f \in \bbV^{I^\sfa}$ and  $g \in \bbV^{I^\sfc}$;  we define 
$$f =_\infty g \text{ iff } f \circ \, \pi_{\sfa}^{\sfa \cup \sfc} =_{U_{\sfa \cup \sfc}}  g \circ \, \pi_{\sfc}^{\sfa \cup \sfc} ,$$
$$f \in_\infty g \text{ iff } f \circ \, \pi_{\sfa}^{\sfa \cup \sfc} \in_{U_{\sfa \cup \sfc}}  g \circ \, \pi_{\sfc}^{\sfa \cup \sfc} .$$
We let $[f]_\infty $ be the equivalence class of $f$ modulo $=_\infty$ (again using Scott's Trick)
and let $\bbV_\infty = \{ [f]_\infty \,\mid\, f \in \bigcup_{\sfa} \bbV^{I^\sfa} \}$ and 
$[f]_\infty \in_\infty [g]_\infty $ iff $ f \in_\infty g$.

The \emph{iterated ultrapower of $\bbV$ by} $U$ is the structure $(\bbV_\infty, \in_\infty)$.
It is the direct limit  of the system of structures
$(\bbV_\sfa,\; \wpi^\sfc_\sfa;\; \sfa, \sfc \in \cP^{\fin}(\bbN), \,\sfa \subseteq \sfc)$.
The mappings $\wpi_{\sfa}^{\infty} : \bbV_\sfa \to \bbV_\infty$ defined by 
$\wpi_{\sfa}^{\infty} ([f]_\sfa) = [f]_\infty$ are elementary embeddings; we identify $\bbV_\sfa$ with its image by this embedding.
In particular, $\wpi_{\emptyset}^{\infty} $ embeds $\bbV :\cong \bbV_{\emptyset}$ elementarily into $\bbV_\infty$.

In addition to the canonical elementary embeddings $\wpi_{\sfa}^{\sfc}$ for $\sfa \subseteq \sfc$, the iterated ultrapower allows other elementary embeddings, due to the fact that the same ultrafilter $U$ is used throughout the construction.
If $|\sfa| = |\sfc|$ and $\al$ is the order-preserving mapping of  $\sfa$ onto $\sfc$, define $\pi_\sfa^\sfc: I^\sfc \to I^\sfa$ by $\pi_\sfa^\sfc(\mathbf{i})= \mathbf{i} \circ \al$ for $\mathbf{i} \in I^\sfc$.
Then $\wpi^\sfc_\sfa $ is an isomorphism of $(\bbV_\sfa, \in_\sfa)$ and $(\bbV_\sfc, \in_\sfc)$.

We fix $r \in \bbN$.
For $f \in \bbV^{I^{r+n}}$ and $\mathbf{i} \in I^r$ we define  $f_{\mathbf{i}}  \in \bbV^{I^n}$ by
$f_{\mathbf{i}}  (\mathbf{j})= f(\mathbf{i}, \mathbf{j})$ for all $\mathbf{j} \in I^n$ and let
$\Omega_n\, ([f]_{r+n}) = [F]_{r}$ where $F(\mathbf{i} )= [f_{\mathbf{i}} ]_{n}$ for all $\mathbf{i} \in I^r$. 
It is routine to check that 
$\Omega_n: \bbV_{r+n} \to  (\bbV_n)^{I^r} /U_r$ 
is well-defined 
 and it
 is an isomorphism of the structures $(\bbV_{r+n}, \in_{r+n}) $ and $(\bbV_n , \in_n)^{I^r}/U_r$.
\

We use the notations $r \oplus \sfa = \{ r + s \,\mid\, s \in \sfa\}$ and 
$r \boxplus \sfa = r \,\cup\, (r \oplus \sfa$).  Note that if $\sfa = n = \{0,\ldots, n-1\} \in \bbN$, then $r \boxplus n = r + n$.

\begin{prop} (Factoring Lemma) \label{P3}
The mapping $\Omega_n$ is an isomorphism of the structures
$$ (\bbV_n,\in_n ,   \bbV_\sfa,  \,\wpi_\sfa^\sfc; \,\sfa,\sfc \subset n, |\sfa| = |\sfc|)^{I^r}/U_r$$ 
and 
 $$ (\bbV_{r+n},\,\in_{r+n} , \,  \bbV_{r \boxplus\sfa}, \,\wpi_{r \boxplus\sfa}^{r \boxplus\sfc}; \;\sfa,\sfc \subset n,  |\sfa| = |\sfc|).$$ 
\end{prop}


\section{\textbf{SPOTS}} \label{spots}

Theories with many levels of standardness have been developed in 
P\'{e}raire,~\cite{Per} (\textbf{RIST})
and
Hrbacek~\cite{H4, H2b} ($\GRIST$).
The characteristic feature of these theories is that the unary standardness predicate $\st(v)$ is subsumed under a binary \emph{relative standardness} predicate $\sr(u,v)$.

The main advantage of  theories with many levels of standardness is that nonstandard methods can be applied to arbitrary objects, not just the standard ones.
For example, the nonstandard definition of the derivative
$$ f'(a)  = \sha \left( \frac{f(a + h) - f(a)}{h}\right)  \; \text{ where } h \text{ is infinitesimal},$$
which in a single-level nonstandard analysis works for standard $f$ and $a$ only, in these theories works for all $f$ and $a$, provided ``infinitesimal'' is understood as ``infinitesimal relative to the level of $f$ and $a$''
and ``$\sha$'' is  ``$\sha$ relative to the level of $f$ and $a$.''
In the book  Hrbacek, Lessmann and O'Donovan~\cite{HLOD}
this approach is used to develop elementary calculus.

Jin's work using multi-level nonstandard analysis goes beyond the features postulated by these theories in that it also employs  nontrivial elementary embeddings (ie, other than those provided by inclusion of one level in a higher level).

\bigskip
The language of $\textbf{SPOTS}$ has  a binary predicate symbol $\in$, a binary predicate symbol $\sr$\;  ($\sr(u,v)$ reads ``$v$ is $u$-\emph{standard}'') and a ternary function symbol $\bs$ that captures the relevant isomorphisms. 
The unary predicate $\st(v) $ stands for $\sr(\emptyset, v)$, 
Variables $\sfa$, $\sfc$ (with decorations)  range over \emph{standard finite} subsets of $\bbN$; we refer to them as \emph{labels}.
We use the class notation
$\bbS_\sfa = \{ x \,\mid\, \sr(\sfa,x)\}$ and $\bbI_{\sfa}^{\sfc} = \{ \la x, y \ra  \,\mid\, \bs(\sfa,\sfc,x) = y\}$.
If $\sfa$ is a standard natural number, we use $\sfn$  instead of $\sfa$; analogously for $\sfc$ and $\sfm$.
We call $\bbS_\sfn$ the $\sfn$-\emph{th level of standardness}.
In particular, $\bbS = \bbS_0 =\{ x \,\mid\, \st(x)\}$ is the universe of standard sets.

As in Section~\ref{ultrapower}, for standard $r \in \bbN$ we let $r \oplus \sfa = \{r+s \,\mid\, s \in \sfa\}$ and $r \boxplus \sfa = r \,\cup\, (r \oplus \sfa)$.
Also $\sfa < \sfc$ stands for $\forall s \in \sfa\;\forall t \in \sfc\;(s<t)$.

A formula $\Phi$ is \emph{admissible} if labels appear in it only as subscripts and superscripts of $\bbS$ and $\bbI$, and all quantifiers are of the form $\forall v \in \bbS_\sfa$ and $\exists v \in \bbS_\sfc$.

\begin{defn} (\emph{Admissible formulas})
\begin{itemize}
\item
$u = v$, $u \in v$, $v \in \bbS_\sfa$ and $\bbI_\sfa^\sfc (u) = v$ are admissible formulas
\item
If $\Phi$ and $\Psi$ are admissible, then $\neg \Phi$, $\Phi \wedge \Psi$, $\Phi \vee \Psi$, $\Phi \to \Psi$ and $\Phi \eqi \,\Psi$ are admissible
\item
If $\Phi$ is admissible, then $\forall v \in \bbS_\sfa \,\Phi$ and $\exists v \in \bbS_\sfc \,\Phi$ are admissible
\end{itemize}
\end{defn}

Let $r$ be a variable that ranges over standard natural numbers and  does not occur in the formula $\Phi$.
 The formula $\Phi^{\uparrow r}$ is obtained from  $\Phi$ by replacing each occurrence of  $\bbS_\sfa$ with $\bbS_{r\boxplus \sfa}$
and each occurrence of $\bbI_\sfa^{\sfc}$ with $\bbI_{r \boxplus\sfa }^{\,r\boxplus\sfc}$.
In particular, if $\Phi$ is a formula where only the symbols $\bbS_\sfn$  for  $\sfn  \in \bbN \cap \bbS_0$ occur, then 
 $\Phi^{\uparrow r}$ is obtained from $\Phi$ by shifting all levels by $r$.

\bigskip
The iterated ultrapower construction described in Section~\ref{ultrapower} suggests the axioms
\textbf{IS}, \textbf{GT} and \textbf{HO}.

\bigskip
\textbf{IS}  (Structural axioms)
\begin{enumerate}
\item
$\sr(u,v) \to \exists \sfa \; (u = \sfa )$, \; $\bs (u,v,x) = y \to \exists \sfa, \sfc \; (u = \sfa \wedge v = \sfc)$, 

$\bbI_\sfa^\sfc(u) = v  \to |\sfa| = |\sfc|$
\item
$\forall x \, \exists \sfa\; (x \in \bbS_\sfa)$
\item
For all $\sfa, \sfc$,  \quad $\bbS_{\sfa \,\cap\, \sfc} = \bbS_{\sfa } \,\cap\, \bbS_{\sfc}$ \; (in particular, $\sfa \subseteq \sfc\to \bbS_\sfa \subseteq \bbS_{\sfc}$)

\item
If  $|\sfa| =  |\sfa' |= |\sfa''|$, then
$$\bbI_\sfa^{\sfa'} : \bbS_\sfa \to \bbS_{\sfa'}, \quad\bbI_\sfa^\sfa = \text{Id}_{\bbS_\sfa}, \quad \bbI_{\sfa'}^\sfa = (\bbI_\sfa^{\sfa'})^{-1}, \quad  \bbI_{\sfa}^{\sfa'} \circ \bbI_{\sfa'}^{\sfa''} =\bbI_\sfa^{\sfa''} $$ 
$$\forall x,z \in \bbS_\sfa \;  ( x \in z \eqi \bbI_\sfa^{\sfa'} (x) \in \bbI_\sfa^{\sfa'} (z) )$$

\item
If $|\sfa| = |\sfa'|$ and $\sfc \subset \sfa$, then
$$ x \in \bbS_\sfc \to      \bbI_\sfa^{\sfa'} (x) =  \bbI_\sfc^{\sfc'} (x),    $$
where  $\sfc'$ is the image of $\sfc$ by the order-preserving map of $\sfa$ onto~$\sfa'$

\end{enumerate}

\bigskip
\textbf{GT} (Generalized Transfer)

Let $\phi(v, v_1,\ldots, v_k)$ be an $\in$--formula. Then for all $\sfa \in \cP^{\fin}(\bbN)\cap \bbS_0$
$$\forall x_1,\ldots, x_k \in \bbS_\sfa\;\left( \forall x \in \bbS_\sfa\; \phi(x, x_1,\ldots, x_k)  \to  \forall x\; \phi(x, x_1,\ldots, x_k)\right).$$

\bigskip
\textbf{HO} (Homogeneous Shift)

Let $\Phi(v_1,\ldots, v_k)$ be an admissible formula. 
For all standard~$r$ and all  $\sfa \in \cP^{\fin}(\bbN)\cap \bbS_0$
$$\forall x_1,\ldots, x_k \in \bbS_\sfa\;[\,\Phi( x_1,\ldots, x_k )  \eqi \Phi^{\uparrow r}(\bbI_\sfa^{r\oplus\sfa}( x_1),\ldots, \bbI_\sfa^{r\oplus\sfa}(x_k) )\,].$$

\bigskip
The language of $\SPOTS$ has an obvious interpretation in the iterated ultrapower described in Section~\ref{ultrapower}:
$\bbS_\sfa$ is interpreted as $\bbV^{I^\sfa}/U_\sfa$ and $\bbI^\sfc_\sfa$ is interpreted as $\wpi^\sfc_\sfa$.

\begin{prop} \label{PP2}
Under the above interpretation, the axioms $\mathbf{IS}$, $\mathbf{GT}$ and $\mathbf{HO}$  hold in the iterated ultrapower constructed in Section~\ref{ultrapower}.
\end{prop}

\begin{proof}
The axiom (3) in $\mathbf{IS}$ follows from Proposition~\ref{P2}; the rest is obvious.

To prove that $\mathbf{GT}$ holds,  we recall that the mapping $\wpi_{\sfa}^{\infty} : \bbV_\sfa \to \bbV_\infty$  is an elementary embedding, ie, if $\psi(v_1,\ldots, v_\ell)$ is any $\in$--formula and 
$x_1,\ldots, x_\ell \in \bbS_\sfa$, then $\psi^{\bbS_\sfa} (x_1,\ldots, x_\ell) \eqi \psi (x_1,\ldots, x_\ell)$.
Using this observation twice shows that $$\forall x \in \bbS_\sfa\; \phi(x, x_1,\ldots, x_k) \to \forall x \in \bbS_\sfa\; \phi^{\bbS_\sfa}(x, x_1,\ldots, x_k) \to\forall x \; \phi (x, x_1,\ldots, x_k).$$

\textbf{HO}  is justified by the Factoring Lemma (take $n$ so that all labels occurring in $\Phi$ are proper subsets of $n$) and \L o\'{s}'s Theorem (specifically, by the fact that for each $n$ the canonical embedding of $ (\bbV_n,\in_n ,   \bbV_\sfa,  \,\wpi_\sfa^\sfc )$ into its ultrapower by $U_r$ is elementary).
\end{proof}

$\mathbf{SPOTS}$ is the theory $\SPOT^+  + \textbf{IS} + \textbf{GT} +\textbf{HO}$, where  
Nontriviality is modified to
 $\exists \nu \in \bbN \cap \bbS_1\; \forall^{\st} n \in \bbN\; (n \ne \nu)$
and \textbf{SN} and \textbf{SF} admit all formulas in the language of $\mathbf{SPOTS}$.

A consequence of $ \textbf{GT}$ is the following proposition.

\begin{prop}\label{gt}
The mapping $\bbI_\sfa^\sfc$ is an elementary embedding of $\bbS_\sfa$ into $\mathbb{I}$ (where $\mathbb{I}$ is the class of all sets) and into $\bbS_{\sfc'}$ for every $\sfc' \supseteq \sfc$.
\end{prop}

\begin{prop} \label{gtg}
Let $\Psi(v, v_1,\ldots,v_k)$ be an admissible formula and let $\Phi(v_1,\ldots,v_k)$ be either $\forall v\, \Psi(v, v_1,\ldots,v_k)$ or 
$\exists v\, \Psi(v, v_1,\ldots,v_k)$. Then $\mathbf{HO}$ holds for $\Phi$.

\end{prop}

\begin{proof}
It suffices to  prove the existential version. 
Fix $r\in \bbN\cap \bbS_0$,   $\sfa \in \cP^{\fin}(\bbN)\cap \bbS_0$ and .$x_1,\ldots,x_k \in \bbS_\sfa$.

We have 
$\exists x\, \Psi(x, x_1,\ldots,x_k)$ if and only if there exists $\sfn$ such that $\exists x\in \bbS_{\sfn} \, \Psi(x, x_1,\ldots,x_k)$ if and only if (by applying $\mathbf{HO}$ to this admissible formula)   there exists $\sfn$ such that
$\exists x\in \bbS_{\sfn +r} \, \Psi^{\uparrow r}(x,\bbI_\sfa^{r\oplus\sfa}( x_1),\ldots, \bbI_\sfa^{r\oplus\sfa}(x_k) )$  if and only if
$\exists x\, \Psi^{\uparrow r}(x,\bbI_\sfa^{r\oplus\sfa}( x_1),\ldots, \bbI_\sfa^{r\oplus\sfa}(x_k) )$.
\end{proof}

Hence   the axioms of $\SPOT$, in particular $\SP$, postulated in $\SPOTS$ only about the level $\bbS_0$, hold there about every level $\bbS_\sfn$.

An important consequence of $\SPOTS$  asserts that  every natural number $k \in \bbS_\sfa$ is either standard or greater than all natural numbers at levels less than $\min \sfa$.

\begin{prop} \emph{(End Extension)} \label{EE}
Let $\sfa \neq \emptyset$ and  $n = \min \sfa \in \bbN$.

Then 
$\forall k \in \bbS_\sfa \cap \bbN\; (k\in \bbS_0 \,\vee\,  \forall m \in \bbS_n\; (m < k)).$
\end{prop}

\begin{proof}
By Proposition~\ref{HKP1}, $\forall m \in \bbN \cap \bbS_0\; \forall k\in \bbN \;(k \le m \to k \in \bbS_0).$
By Proposition~\ref{gtg} this implies
$\forall m \in \bbN \cap \bbS_n \; \forall k \in \bbN \;(k \le m \to k \in \bbS_n ).$
If $k \in \bbS_\sfa \cap \bbN$ and $\exists m \in \bbS_\mathsf{n}\; (m \ge k)$
then \; $k\in \bbS_n$ by the above. As $\sfa \cap n = \emptyset$, we get $k \in \bbS_0$.
\end{proof}

$\mathbf{SCOTS}$ is the theory $\SCOT  + \textbf{IS} + \textbf{GT} +\textbf{HO} $ = $\mathbf{SPOTS} + \textbf{DC}$, where the axiom schema 
\textbf{DC} is formulated as follows.

\bigskip
\textbf{DC}  (Dependent Choice)  \emph{ Let $\Phi (u,v)$ be a formula in the language of $\SPOTS$, with arbitrary parameters. 
For any $\sfa$:}\\
\emph{
If  $\forall x \in \bbS_\sfa\, \exists y \in \bbS_\sfa\; \Phi(x,y)$, then for every $b \in  \bbS_\sfa$  there is a sequence $\bar{b} = \la b_n \,\mid\, n \in \bbN\ra \in \bbS_\sfa$ such that $b_0 = b $ and }
$\forall n \in \bbN \cap \bbS_0\; \Phi(b_n, b_{n+1})$.
 
\bigskip
$\DC$ implies  Countable Standardization (and hence $\textbf{SF}$). 

$\SC$ (Countable Standardization)  \emph{Let $\Psi(v)$ be  a formula in the language of $\SPOTS$, with arbitrary parameters.  Then}
$$
\exists S\in \bbS_0\; \forall x \in \bbS_0\; (x \in S \eqi x \in \bbN
\,\wedge\, \Psi(x)).
$$

\begin{thm}\label{maintheorem}
$\mathbf{SCOTS}$ is a conservative extension of $\ZF + \ADC$. 
 \end{thm}

\begin{thm}\label{maintheorem2}
$\mathbf{SPOTS }$ is conservative over $\ZF + \ACC$. 
 \end{thm}

The proofs are given in Section~\ref{conservativity}.
\begin{conjecture}
$\mathbf{SPOTS}$ is a conservative extension of $\ZF$. 
 \end{conjecture}


\section{Jin's proof of Ramsey's Theorem in \textbf{SPOTS}} \label{ramsey}

\begin{rlem}
Given a coloring $c:[\bbN]^n \to r$ where $n, r \in \bbN$, there exists an infinite set $H \subseteq \bbN$ such that $c \uhr [H]^n$ is a constant function. 
\end{rlem}

We formalize in \textbf{SPOTS} the proof presented by Renling Jin~\cite{RJ2} in his invited talk at the conference
\emph{Logical methods in Ramsey Theory and related topics}, Pisa, July 9 -- 11, 2023.
It is included here with his kind permission.

\begin{proof}
It suffices to prove the theorem under the assumption  that $n, r, c$ are standard;
the general result then follows by Transfer.

Let $\bbI = \bbI_{\{0,1,\ldots,n-1\}}^{\{1,2,\ldots, n\}}$.
Fix $\nu \in \bbN \cap (\bbS_1 \setminus \bbS_0)$ and.
define the $n$-tuple $\bar{x} = \la x_1,\ldots, x_n\ra$  by
$x_1 = \nu$, \; $x_{i+1} = \bbI(x_i)$  for $i = 1,2,\ldots, n-1$
(the existence of $\bar{x}$ is justified by \textbf{SF}).
Let $c_0 = c(\bar{x})$.

 Define a strictly increasing sequence $\{ a_m\}_{m=1}^{\bullet} \subseteq \bbN$, where  $\bullet \in \bbN$ or $\bullet = \infty$, recursively,
using the notation $A_m = \{ a_1,\ldots, a_m\}$ (also $a_0 = 0$ and $A_0 = \emptyset$):

$a_{m+1} = \text{ the least } a \in \bbN \text{ such that } a > a_m \,\wedge\,  c\uhr [A_m \cup \{a\}  \cup \bar{x}]^n = c_0$\\  if such $a$ exists; otherwise $a_{m+1} $ is undefined and the recursion stops.

Let  $A = \bigcup_{m=1}^{\bullet}  A_m$. Then $A$ is a set and by $\SP$ there is a standard set $H$ such that $\forall^{\st} x\; ( x \in H \eqi x \in A)$.
Clearly  $c \uhr [H]^n = c_0$.

It remains to prove that $H$ is infinite, ie, that $a_m$ is defined and standard for all standard $m \in \bbN \setminus \{0\}$.

Fix a standard $m \in \bbN$. The sentence 
$$\exists x \in \bbN \cap \bbS_1\; \left( x > a_m \,\wedge\, c \uhr [ A_m \cup \{x, \bbI(x_1),\ldots,\bbI(x_{n-1} )\} ]^n = c_0 \right)$$
is true (just let $x = x_1$).

By \textbf{HO},
$\exists x \in \bbN \cap \bbS_0\; \left( x > a_m \,\wedge\, c \uhr [ A_m \cup \{x, x_1,\ldots,x_{n-1} \} ]^n = c_0 \right).$
Let $a_{m+1} $ be the least such $x$ and note that it is standard.

We have 
$c \uhr [ A_{m+1} \cup \{x_1,\ldots,x_{n-1} \} ]^n = c_0$. 
It remains to show that 
$c \uhr [ A_{m+1} \cup \{x_1,\ldots, x_{n-1}, x_{n} \} ]^n = c_0$.

Consider $\bar{b} = \{b_1 <\ldots<b_n\} \in   [ A_{m+1} \cup \{x_1,\ldots, x_{n-1}, x_{n} \} ]^n $.

If $b_n < x_n$ then $b_n \le x_{n-1}$ and $c(\bar{b} )= c_0$.

If $b_1 = x_1$ then $\bar{b} = \bar{x}$ and $c(\bar{b} )= c(\bar{x})) =c_0$.

Otherwise $b_1 \in \bbN \cap \bbS_0$ and $b_n = x_n$.
Let $p$ be the largest value such that $x_p \notin \bar{b}$ (clearly $1 \le p < n$) and
let $\mathbf{J} = \bbI_{\{0,\ldots, n-1\}}^{\{0,\ldots,p-1,p+1,\ldots,n\}}$.

Note that 
$\mathbf{J} (b_j) = b_j $ for $j \le p$,
$b_j = x_j$, and $\mathbf{J} (b_j) = \mathbf{J}(x_j) = x_{j+1}$ for $p < j \le n-1$
(because  $\bbI_{\{p,\ldots, n-1\}}^{\{p+1,\ldots,n\}} \subseteq \bbI, \mathbf{J}$, ie., $\bbI$ and $\mathbf{J}$ agree on $\bbS_{\{p,\ldots, n-1\}}$).
Let $\bar{b}' = \mathbf{J}^{-1}(\bar{b})$.
Then $\bar{b}' \in [A_{m+1} \cup \{x_1,\ldots,x_{n-1}\}]^n$, hence $c(\bar{b}') = c_0$. 
By \textbf{HO} shift via $\mathbf{J}$, $c(\bar{b}) = c_0$.
\end{proof}


\section{Jin's proof of Szemer\'{e}di's theorem in \textbf{SPOTS}} \label{szemeredi}

Jin's proof in~\cite{RJ} uses four universes ($\bbV_0, \bbV_1, \bbV_2$ and $\bbV_3$)  and some additional elementary embeddings.  
Let $\bbN_j = \bbN \cap \bbV_j$ and $\bbR_j = \bbR \cap \bbV_j$  for $j = 0,1,2,3$.
Jin summarizes the required properties of these universes:

\bigskip
0.
$\bbV_0 \prec \bbV_1 \prec \bbV_2 \prec \bbV_3$.

1.
$\bbN_{j+1}$ is an end extension of $\bbN_{j}$ \;($j = 0,1,2$).

2.
For $j' > j$, Countable Idealization holds from $\bbV_{j}$ to $\bbV_{j'}$:
Let $\phi$ be an $\in$--formula with parameters from $\bbV_{j'}$. Then
$$\forall  n \in \bbN_j\;\exists x  \in \bbV_{j'}\; \forall m \le n \;   \phi^{\bbV_{j'}} (m,x)\eqi 
\exists x\in \bbV_{j'} \; \forall n \in \bbN_j \; \phi^{\bbV_{j'}}(n,x).$$

3.
There is an  elementary embedding $i_{\ast}$ of 
$(\bbV_2; \bbR_0,\bbR_1)$ to $(\bbV_3; \bbR_1,\bbR_2)$.

4.
There is an elementary embedding $i_1$ of $(\bbV_1, \bbR_0)$ to $(\bbV_2, \bbR_1)$ such that
$i_1 \upharpoonright \bbN_0$ is an identity map and $i_1(a) \in \bbN_2 \setminus \bbN_1$ for each $a \in \bbN_1 \setminus \bbN_0$.

5.
There is an elementary embedding $i_2$ of $\bbV_2$ to $\bbV_3$ such that
$i_2 \upharpoonright \bbN_1$ is an identity map and $i_2(a) \in \bbN_3 \setminus \bbN_2$ for each $a \in \bbN_2 \setminus \bbN_1$.

\bigskip
These requirements are listed as Property 2.1 in arXiv versions v1, v2 of Jin's paper, and appear in a slightly different form in Section~2 of the Discrete Analysis version; see especially Property  2,7 there.
Our formulations  differ from his in two significant ways.  
\begin{itemize}
\item
Jin works model-theoretically and his universes are superstructures, that is, sets of $\ZFC$.
In contrast, our universes are proper classes. Nonstandard arguments work similarly in both frameworks.
\item
In Property 2 Jin postulates Countable Saturation, while the weaker Countable Idealization stated here is more suited for the axiomatic approach. In all instances where Property 2 is used in Jin's proof, Countable Idealization suffices.
\end{itemize}

\begin{prop}
$\mathbf{SPOTS}$ interprets Jin's Properties 0. -- 5.
\end{prop}

\emph{Proof.}
We  define:
$\bbV_0 = \bbS_0$, $\bbV_1 = \bbS_{\{0\}}$, $\bbV_2 = \bbS_{\{0,1\}}$, 
 $\bbV_3 = \bbS_{\{0,1,2\}}$, and \\
$i_1 = \bbI_{\{0\}}^{\{1\}}$, $i_2= \bbI_{\{0,1\}}^{\{0,2\}}$,  $i_{\ast} = \bbI_{\{0,1\}}^{\{1,2\}}$.

Property 0  follows from $\textbf{GT}$, and Property 1 from Proposition~\ref{EE}.

Property 2.
Countable Idealization is a consequence of $\SPOT$, so it suffices to show that each $(\bbS_{j'}, \in, \bbS_j)$
satisfies the axioms of $\SPOT$. The axiom $\SP$ is the only issue.

$\SP$ holds in $(\mathbb{I}, \in, \bbS_0)$, hence it holds in every $(\mathbb{I}, \in, \bbS_j)$ by \textbf{HO}.
Its validity in $(\bbS_{j'}, \in, \bbS_j)$ follows.

Property 3.
If $\psi(v_1,\ldots, v_r)$ is a formula in the common language of the structures $(\bbS_2, \in, \bbS_0, \bbS_1)$ and $(\bbS_3, \in, \bbS_1, \bbS_2)$, then,  by \textbf{HO}, 
$$\forall\,x_1,\ldots,x_r \in \bbS_2\;[\psi^{\bbS_2} (x_1, \ldots, x_r) \eqi 
\psi^{\bbS_3} (\bbI_{\{0,1\}}^{\{1,2\}}(x_1), \ldots, \bbI_{\{0,1\}}^{\{1,2\}}(x_r)) ].$$
Properties 4. and 5. follow from Propositions~\ref{gt} and~\ref{EE} and the observation that 
$i_1 = i_\ast \upharpoonright \bbV_1$.
\qed

\bigskip
It remains to  show that $\SPOT$ proves the existence  of densities used  by  Jin.  This requires a careful appeal to Standardization.

\begin{defn}\label{densities}
In our notation:
\begin{enumerate}
\item
For  finite  $A\subseteq \bbN$ with $|A|$ unlimited, the \emph{strong upper Banach density of} $A$ is defined by
$$SD(A) =\sup{}^{\st}   
\{\sha( |A \cap P|/ |P| ) \,\mid\, |P|\text{ is unlimited}\}.$$

\item
If $S \subseteq \bbN$  has $SD(S) = \eta \in\bbR$ (note $\eta$ is standard) and $A \subseteq  S$, the \emph{strong upper Banach density  of} $A$ \emph{relative to} $S$ is defined by
$$SD_S(A) =\sup{}^{\st}   
\{\sha( |A \cap P|/ |P| ) \,\mid\, |P|\text{ is unlimited}
\,\wedge\, \sha(|S \cap P|/|P|) = \eta \}.$$ 
\end{enumerate}
\end{defn}

$\SPOT$ does not prove the  existence of the standard sets of reals whose supremum needs to be taken
(it does not allow Standardization over the uncountable set $\bbR$), but for the purpose of obtaining the supremum, a set of reals can be replaced by a set of rationals.

\begin{prop}\label{existdens}
$\SPOT$ proves the existence of SD$_S(A)$.
\end{prop}

\begin{proof}
We note that
$SD_S(A) =\sup{}^{\st}   
\{q \in \bbQ \,\mid\, \Phi(q)\}$  
where $\Phi(q)$ is the formula
$$\exists P\, [\,  \forall^{st}_\bbN i \,(|P| > i)
\,\wedge\, \forall^{st}_\bbN j (|\,|S \cap P|/|P| - \eta | < \tfrac{1}{j+1}) \,\wedge\, 
q \le |A \cap P|/ |P| \,] .$$

The formula $\Phi$ is equivalent to 
$$\exists P\, \forall^{st}_\bbN i\, [\,   (|P| > i)
\,\wedge\,  (|\,|S \cap P|/|P| - \eta | < \tfrac{1}{i+1}) \,\wedge\, 
q \le |A \cap P|/ |P| \,] ,$$
which, upon the exchange of the order of $\exists P$ and 
$\forall^{st}_\bbN i$, enabled by Countable Idealization, converts to an $\bbN$- special formula 
$$\forall^{st}_\bbN i\, \exists P\,  [\,   (|P| > i)
\,\wedge\,  (|\,|S \cap P|/|P| - \eta | < \tfrac{1}{i+1}) \,\wedge\, 
q \le |A \cap P|/ |P| \,] .$$
Proposition~\ref{cssf} concludes the proof.
\end{proof}

The definitions of these densities  relativize to every level $j > 0$.  Their existence at higher levels  follows from Proposition 6.3 and the observation (in the proof of Proposition 6.1, Property 2) that $(\bbS_{j'}, \in , \bbS_j)$ satisfies $\SPOT$.


\section{Conservativity} \label{conservativity}

Conservativity of $\SPOT$ over $\ZF$  was established in Hrbacek and Katz~\cite{HK}  by a construction that  extends and combines the methods of forcing developed by Ali Enayat~\cite{E}  and Mitchell Spector~\cite{Spr}.
Conservativity of $\SCOT$  over $\ZF + \ADC$ is obtained there as a corollary. Here we give a simple, more direct proof of the latter result that generalizes straightforwardly to the proof of conservativity of $\mathbf{SCOTS}$ over $\ZF + \ADC$.

We prove the following proposition.

\begin{prop}\label{mainprop}
Every countable model $\cM = (M, \in^{\mathcal{M}})$ of $\ZF + \ADC$
has an extension to a model of $\SCOTS$ in which elements of $M$ are exactly the standard sets.
\end{prop}

The difficulty is that $\cM$ may contain no nonprincipal ultrafilters. We add such an ultrafilter to $\cM$ by forcing, and then carry out the construction of the iterated ultrapower as in Section~\ref{ultrapower} inside this generic extension of $\cM$.

Jech~\cite{J} is the standard reference for forcing and generic extensions of well-founded models of $\ZF$. For details on the extension of this material to non-well-founded models see Corazza~\cite{Cor, Cor2}.


\subsection{Forcing}\label{s1}
In this subsection we work in $\ZF + \ADC$. 

\begin{defn}
Let $\bbP = \{ p \subseteq \omega \mid p \text{ is infinite}\}$. 
For $p, p' \in \bbP$ we say that $p'$ \emph{extends} $p$ (notation: $p' \le p$) if $p' \subseteq p$.
\end{defn}

The poset $\bbP$ is not \emph{separative} (Jech~\cite[Section 17]{J}); forcing with $\bbP$  is equivalent to forcing with $\widetilde{\bbP} = \cP^{\infty}(\om)/_{\fin}$.

The poset $\widetilde{\bbP} $ is $\om$--\emph{closed}: If  $\la p_n \,\mid\, n \in \om\ra$ is a sequence of conditions from $\bbP$ such that, for each $n \in \om$,  $p_{n+1} \setminus p_n $ is finite, then there is $p \in \bbP$ such that 
$p \setminus p_n$ is finite for all $n \in \om $.
It follows that the forcing with $\bbP$ does not add any new countable sets  (note that the proof of this fact uses $\ADC$).

The forcing notion $\bbP$ is \emph{homogeneous} in the sense that 
for $x_1,\ldots, x_s \in \bbV$ and $p, p' \in \bbP$ we have 
$p \Vdash \phi(\check{x}_1, \ldots, \check{x}_s)$ iff $p' \Vdash \phi(\check{x}_1, \ldots, \check{x}_s)$.
(Jech~\cite[Lemma 19.10 and related material]{J}.)

This is a consequence of the following fact (we let $p^c = \bbN \setminus p$):  For all $p_1, p_2 \in \bbP$ such that  $p^c_1, p^c_2$ are infinite, there is an automorphim $\pi$ of $\bbP$ such that $\pi (p_1) = p_2$.
It can be obtained as follows: Fix a one-one mapping $\al$ of $\om$ onto $\om$ such that $\al$ maps $p_1 $ onto $p_2$ in an order-preserving way, and maps $p^c_1 $ onto $p^c_2$ in an order-preserving way, and then define
$\pi (p) = \al [p]$.

\subsection{Generic Extensions}\label{s2}
Let $\cM = (M, \in^{\mathcal{M}})$ be a countable model of $\ZF + \ADC$
and let $\cG$ be an $\cM$-generic filter over $\bbP^{\cM}$.
The generic extension $\cM[\cG]$ is a model of $\ZF + \ADC$ extending $\cM$ and 
the forcing does not add any new reals or countble subsets of $M$, ie,
every countable subset of $M$ in $\cM[\cG]$ belongs to $M$.

We need the following observation: The structure $(M[\cG], \in^{\mathcal{M}[\cG]}, M)$ is a model of $(\ZF + \ADC)^{\bbM}$, a theory obtained by adding a unary predicate symbol $\bbM$ to the $\in$--language of $\ZF$ and postulating that the axioms of Separation, Replacenment and Dependent Choice hold for formulas in this extended language. 
This is a piece of folklore; a proof can be given by adding the predicate $\bbM$ to the forcing language and defining  
$$p \Vdash \bbM(x) \text{ iff } \forall p' \le p\, \exists p'' \le p' \, \exists z  \,(p'' \Vdash x = \check{z}).$$ 
One can then prove the appropriate versions of Forcing Theoerm and  the Generic Model Theorem as in Jech~\cite[Section 18]{J}.


\subsection{Conservativity of $\SCOTS$ over $\ZF + \ADC$.}\label{s3}
We work in the structure $(M[\cG], \in^{\mathcal{M}[\cG]}, M)$, a model of $(\ZF + \ADC)^{\bbM}$, and use $\om$ to denote its set of natural numbers.
The generic filter $\cG$ is a nonprincipal ultrafilter over $\om$ and one can construct the expanded iterated ultrapower 
$$\cM_\infty = (\bbM_\infty, \in_\infty, \bbM_\sfa, \Pi_\sfa^\sfc; \sfa, \sfc \in \cP^{\fin}(\om), |\sfa| = |\sfc|)$$ 
 of $\bbM$ by $\cG$
as in Section~\ref{ultrapower} (let $I = \om$, $U = \cG$, and  replace $\bbV$ by $\bbM$). 

\L o\'{s}'s  Theorem holds 
because $\ACC$ is available, and 
$\Pi_0^\infty$ canonically embeds $\bbM$ into $(\bbM_\infty, \in_\infty)$. 
 The structure $\cM_\infty$ interprets the language of $\SPOTS$ (with $\bbS_\sfa$ interpreted as $\bbM_\sfa$ and $\bbI_\sfa^\sfc$ interperted as $ \Pi_\sfa^\sfc$).
As in Proposition~\ref{PP2}, the structure $\cM_\infty$ satisfies 
$\textbf{IS}$, $\textbf{GT}$ and $\textbf{HO}$.
It remains to show that $\SN$ and $\DC$ hold there.

\begin{prop}
$\DC$ holds in $\cM_\infty$.\label{Prop2}
\end{prop}

\begin{proof}

Let $\Phi (u,v,w)$ be a formula in the language of $\SPOTS$.
Let $b \in \bbM_\sfa$ and $c \in \bbM_\infty$ be such that  
$$\Psi(B,b,c) : \quad [ b \in \bbS_\sfa\,\wedge\, \forall x \in  \bbS_\sfa\; \exists y \in  \bbS_\sfa\; \Phi(x,y,c)]^{\cM_\infty}$$ holds. 
(The superscript $\cM_\infty$ indicates that the quantifiers range over $\bbM_\infty$ and all symbols are interpreted in $\cM_\infty$.)
$\Psi$ is (equivalent to) a formula of the forcing language (we identify $b,c$ and $\sfa$ with their names in the forcing language), hence there is $p \in \bbP \cap \cG$ such that $p \Vdash \Psi$.
Let $p_0 \le p$.

We let the variable $a$ (with decorations) range over the names in the forcing language and define the class
$$\mathbf{A} = \{ \la p', a' \ra \,\mid\, p' \le p_0 \,\wedge\, p' \Vdash [a' \in\bbS_\sfa]^{\cM_\infty}\}.$$
Note that $\la p_0, b \ra \in \mathbf{A} $,   and define $\mathbf{R}$ on $\mathbf{A} $ by
$$ \la p', a' \ra \mathbf{R} \la p'', a'' \ra \text{ iff } p'' \le p' \,\wedge\, p'' \Vdash \Phi^{\cM_\infty}(a', a'',c).$$

It is clear from the properties of forcing that for every $ \la p', a' \ra \in \mathbf{A} $ there is $\la p'', a'' \ra  \in \mathbf{A} $ such that $ \la p', a' \ra \mathbf{R} \la p'', a'' \ra $. Using Reflection and $\ADC$ we obtain a sequence 
$\la \la p_n, a_n \ra \,\mid\, n \in \om \ra$ such that  $a_0 = b$, and, for all $n \in \om$, $\la p_n, a_n\ra \in A$, $p_{n+1} \le p_n$ and $p_{n+1} \Vdash \Phi^{\cM_\infty}(a_n, a_{n+1},c)$.

As the forcing is $\om$--closed, one obtains $p_\infty \in \bbP$ and $\la a_n \mid n \in \om\ra$ such that $p_\infty \le p_0$ and 
$p_\infty \Vdash  [a_n \in  \bbS_\sfa\,\wedge\, \Phi(a_n, a_{n+1},c)]^{\cM_\infty}$ for all $n \in \om$.

By the genericity of $\cG$ there is some $p_\infty \in \cG$ and the associated sequence $\la a_n \mid n \in \om\ra$  with this property.
Hence $(M[\cG], \in^{\mathcal{M}[\cG]}, M)$ satisfies 
$[a_n \in \bbS_\sfa\,\wedge\, a_0 = b \,\wedge\,\Phi(a_n, a_{n+1},c)]^{\cM_\infty}$ for all $n \in \om$.

The class $\bbS_\sfa$ is interpreted in $\cM_\infty$ by 
the ultrapower $\bbM_\sfa = \bbM^{I_\sfa}/ U_\sfa$ (for $U = \cG$) .
Since this ultrapower is $\om_1$--saturated, there is 
$\bar{b} \in \bbM_\sfa$ such that $$[ \bar{b} \text{ is a function } \,\wedge\, \dom \bar{b} = \bbN \,\wedge\, b_n =a_n]^{\bbM_\sfa}$$ holds for every $n \in \om$.
This translates to the desired
$$[ \bar{b}  \in \bbS_{\sfa} \,\wedge\,  \dom \bar{b} = \bbN  \,\wedge\, b_0=b \,\wedge\, \forall n \in \bbN \cap \bbS_0\;\Phi(b_n, b_{n+1},c)]^{\cM_\infty}.$$
\end{proof}

\begin{prop}
$\SN$ holds in $\cM_\infty$.
\end{prop}

\begin{proof}

Let $\Phi (u)$ be a formula in the language of $\SPOTS$ (with no parameters) and $A\in \bbM$.
Let $\Psi(u)$ be the formula $\Phi(u)^{\cM_\infty}$ of the forcing language.
By homogeneity of the forcing, $p \Vdash \Psi(\check{a})$ iff $p' \Vdash \Psi(\check{a})$ holds for all $a \in A$ and $p, p' \in \bbP$. Fix some $p \in \cG$ and let $S = \{ a \in A \,\mid\, p \Vdash \Psi(\check{a})\}$.
For $a \in \bbS$ then $a \in S$ iff $a \in A \,\wedge\, \Phi(a)^{\cM_\infty}$  holds.
\end{proof}

The structure $\cM_\infty$ is a class model of $\SCOTS$ constructed inside the countable model
$(M[\cG], \in^{\mathcal{M}[\cG]}, M)$.
It converts into a countable model $\widetilde{\cM}_\infty$ in the meta-theory so that 
$\Phi^{\cM_\infty}  \eqi \widetilde{\cM}_\infty \vDash \Phi$ for all formulas in the language of $\SCOTS$.


\subsection{Finitistic proofs}
\label{finitistic}
The model-theoretic proof of Proposition~\ref{mainprop} in Subsections~\ref{s1} -- \ref{s3} is carried out in $\ZF$. 
Using techniques from Simpson~\cite[Chapter II, especially II.3 and II.8]{Si}, it can be verified that the  proof goes through in $\textbf{RCA}_0$ (wlog one can assume that $M \subseteq \omega$).

The proof of Theorem~\ref{maintheorem} from Proposition~\ref{mainprop} requires the G\"{o}del's Completeness Theorem and therefore $\textbf{WKL}_0$; see~\cite[Theorem IV.3.3]{Si}.
We conclude that Theorem~\ref{maintheorem} can be proved in $\textbf{WKL}_0$. 

Theorem~\ref{maintheorem}, when viewed as an arithmetical statement resulting
from identifying formulas with their G\"{o}del numbers, is
$\Pi^{0}_{2}$.  It is well-known that $\textbf{WKL}_0$ is conservative
over $\textbf{PRA}$ (Primitive Recursive Arithmetic) for $\Pi^{0}_{2}$
sentences (\cite[Theorem IX.3.16)]{Si}; therefore Theorem~\ref{maintheorem}
is provable in $\textbf{PRA}$.  The theory $\textbf{PRA}$ is generally
considered to correctly capture finitistic reasoning  (see eg Simpson~\cite[Remark IX.3.18]{Si}).
We conclude that Theorem~\ref{maintheorem} has a finitistic proof.

These remarks apply equally to Theorem~\ref{maintheorem2} and Proposition~\ref{sf}.



\subsection{Conservativity of $\SPOT^+$ over $\ZF$.}
The forcing construction used to establish conservativity of $\SPOT^+$  over $\ZF$ is much more complicated because one needs to force both a generic filter $\cG$  and the validity of
\L o\'{s}'s Theorem in the corresponding ``extended ultrapower.''
We describe the appropriate forcing conditions (see~\cite{HK}).

Let $\bbQ = \{  q \in \bbV^{\om} \mid\, 
\exists k \in \om \; \forall i \in \om \, (q(i) \subseteq \bbV^{k} \,\wedge\, q(i) \neq \emptyset )\}$. 

The number $k$ is the \emph{rank} of $q$. We note that $q(i)$ for each $i \in \om$, and $q$ itself, are sets, but $\bbQ$ is a proper class.

The forcing notion $\bbH$ is defined as follows: $\bbH = \bbP \times \bbQ$ and $\la p', q'\ra \in \bbH$ \emph{extends} $\la  p, q\ra \in \bbH$ iff $p'$ extends $p$, $\rank q' = k' \ge k = \rank q$,
and for almost all $i \in p'$ and all $\la x_0,\ldots, x_{k'-1} \ra \in q'(i)$,  $\la x_0,\ldots,x_{k-1}\ra \in q(i)$.

The forcing with $\bbH$ adds many new reals; in fact, it makes all ordinals countable.

\begin{prop} \label{sf}
$\mathbf{SPOT}^+$ is a conservative extension of $\ZF$.
\end{prop}

\begin{proof}
Conservativity of $\SPOT + \SN$ over $\ZF$ is established in~\cite[Theorem B]{HK} via forcing with $\bbH$.
It remains only to show that $\mathbf{SF}$ also holds in the model constructed there.

In~\cite[Definition 4.4]{HK}  forcing is defined for $\in$--formulas only, but the definition can be extended to $\st$--$\in$--formulas by adding the clause

(11) $\la p, q \ra  \Vdash \st(\dot{G}_{n})$  iff 
$\rank q = k > n$ and
$$\exists x\; \forall^{\infy} i \in p \; \forall \la x_0,\ldots, x_{k-1}\ra \in q(i)\;( x_{n}= x).$$

\cite[Proposition 4.6 (``\L o\'{s}'s Theorem'')]{HK}  does not hold for $\st$--$\in$--formulas, but 
the equivalence of clauses (1) and (2) in \cite[Proposition 4.12 (The Fundamental Theorem of Extended Ultrapowers)]{HK},  remains valid ($\mathfrak{N }$ is the extended ultrapower of $\cM$):

\emph{Let $ \Phi(v_1,\ldots,v_s )$ be an $\st$--$\in$--formula with parameters from $M$. 
If  $G_{n_1},\ldots,G_{n_s} \in \mathfrak{N}$, then  the following statements are equivalent:}
\begin{enumerate}
\item
$\mathfrak{N} \vDash \Phi(G_{n_1},\ldots,G_{n_s} )$
\item
\emph{There is some $\la p, q \ra \in \cG$ such that $\la p, q \ra  \Vdash \Phi( \dot{G}_{n_1},\ldots,\dot{G}_{n_s} )$
holds in $\cM$}.
\end{enumerate}

We now prove that $\mathbf{SF}$ holds in $\mathfrak{N} $.

Wlog we can assume $A = N \in \om$. 
For every $\la p, q \ra \in \bbH$  and every $n \in N$ there is $\la p' g'\ra \le \la p, q \ra$ such that 
$\la p' g'\ra$ decides $\Phi(\check{n})$.
By induction on $N$, for every $\la p, q \ra \in \bbH$ there is $\la p_N, q_N\ra  \le \la p, q \ra$  that decides $\Phi(\check{n})$ for all $n \in N$ simultaneously.
Hence there is $\la \widetilde{p} , \widetilde{q}  \ra  \in \cG$ with this property.
Let $S = \{n \in N \,\mid\, \la \widetilde{p} , \widetilde{q}  \ra \Vdash \Phi(\check{n})\}$.
By the Fundamental Theorem, $S = \{ n \in N \,\mid\, \mathfrak{N} \vDash \Phi(n)\}$.
\end{proof}

\textbf{Final Remark.}
Labels $\sfa, \sfc$ in $\SPOTS$ range over standard finite sets. This implies that the levels of standardness are enumerated by standard natural numbers. It is an open question whether one could allow labels to range over all finite sets, ie, to have levels of standardness indexed by all natural numbers.  
Theories of this kind have been developed in Hrbacek~\cite{H2b} on the basis of $\ZFC$. It seems likely that the present work could be generalized analogously.


\section*{Acknowledgments} I am grateful to Mikhail Katz for many helpful comments and suggestions.


\bibliographystyle{jloganal}

\end{document}